\documentclass{article}
\usepackage{amsfonts, amssymb, latexsym, amscd}
\usepackage{amsmath, amsthm, epsfig, color, longtable}
\newtheorem{thm}{Theorem}[section]

\newtheorem{lem}[thm]{Lemma}
\newtheorem{prp}[thm]{Proposition}
\newtheorem{asm}[thm]{Assumption}
\newtheorem{ccl}[thm]{Conclusion}

\def\1{{\bf{1}}}
\def\s{{\widetilde{\1}}}
\def\0{{\bf{0}}}
\def\x{{\bf{x}}}
\def\y{{\bf{y}}}
\def\z{{\bf{z}}}
\def\vv{{\bf{k}}}
\def\to{{\!\top}}
\def\SG{{\mathcal{G}}}

\def\J{\widetilde{J}}
\def\L{\widetilde{L}}

\begin{document}

\title{\bf The signed graphs with two eigenvalues unequal to $\pm 1$}

\author{
Willem H. Haemers\thanks{\small{\tt{haemers@uvt.nl}}}
\\
{\small Dept. of Econometrics and O.R., Tilburg University, The Netherlands}
\\[5pt]
Hatice Topcu\thanks{\small{\tt{haticekamittopcu@gmail.com}}}
\\
{\small Dept. of Mathematics, Nev\c{s}ehir Hac{\i} Bekta\c{s} Veli University, T\"urkiye}\\
}
\date{}

\maketitle

\begin{abstract}
\noindent
We complete the determination of the signed graphs for which the adjacency matrix has all 
but at most two eigenvalues equal to $\pm 1$.
The unsigned graphs and the disconnected, the bipartite and the complete signed graphs with
this property have already been determined in two earlier papers.
Here we deal with the remaining cases.
\\[5pt]
\end{abstract}

\section{Introduction}

A {\em signed graph} $G^\sigma$ is a graph $G=(V, E)$ together with
a function $\sigma : E \rightarrow \{-1, +1\}$, called the \textit{signature function}.
So, every edge is either positive or negative.
The graph $G$ is called the {\em underlying graph} of $G^\sigma$.
The adjacency matrix $A$ of $G^\sigma$ is obtained from the adjacency matrix of $G$,
by replacing $1$ by $-1$ whenever the corresponding edge is negative.
The signed graph $G^{-\sigma}$ with adjacency matrix $-A$ is called the {\em negative} of $G^\sigma$.
The spectrum of $A$ is also called the spectrum of the signed graph $G^\sigma$.
For a vertex set $X\subset V$, the operation that changes the sign of all edges between $X$ and $V\setminus X$ is called switching.
In terms of the matrix $A$, switching multiplies the rows and columns of $A$ corresponding to $X$ by $-1$.
If a signed graph can be switched into an isomorphic copy of another signed graph, the two signed graphs are called
{\em switching isomorphic}.
Switching isomorphic signed graphs have similar adjacency matrices and therefore they are cospectral (that is, they
have the same spectrum).
For this and more background on signed graphs we refer to~\cite{BCKW}.

In this paper we determine all signed graphs for which the adjacency matrix has all but at most 
two eigenvalues equal to $\pm 1$.
The unsigned graphs with this property have been determined in \cite{CHVW}.
This made it possible to find all (unsigned) graphs in this class which are 
determined by its adjacency spectrum.
This result motivated us to determine the signed graphs in the class.
The first results of this project are published in \cite{HT}.
There we dealt with the disconnected, the bipartite and the complete signed graphs.
This includes the signed graphs with just one or no eigenvalues unequal to $\pm 1$.
Here we determine the remaining cases.

The extension of spectral characterizations from unsigned to signed graphs
has a complication because the spectrum is invariant under switching.
Thus one can only expect a signed graph to be determined by the spectrum up to switching.
Nevertheless, there do exist some achievements.
For example in \cite{AHMP} it is determined for which $n$ the switching class of the 
signed path $P_n$ is determined by the spectrum.

We define $\SG$ to be the set of signed graphs for which the spectrum has all but at most 
two eigenvalues equal to $1$ or $-1$.
Then $\SG$ is closed under switching, taking the negative, and adding or deleting isolated edges.
Clearly every signed graph switching isomorphic to an unsigned graph in $\SG$ is in $\SG$, 
and so is its negative.
However, there are many other signed graphs in $\SG$.

We use eigenvalue interlacing, spectral properties of equitable
partitions and other techniques from linear algebra for which we refer to~\cite{BH}.
Some background on signed graphs can be found in~\cite{BCKW}.
As usual, $J$ is the all-ones matrix and $O$ the all-zeros matrix.
The all-ones and all-zeros vector are denoted by $\1$ and $\0$ respectively.
When necessary we give the size of $J$, $O$, $\1$ or $\0$ as a subscript.
The reverse identity matrix of order $m$ is denoted by $R_m$ 
(that is, $R_m(i,j)=1$ if $i+j=m+1$, and $R_m(i,j)=0$ otherwise).
Note that all eigenvalues of $R_m$ are equal to $\pm 1$.

\section{Main result}\label{main}

Here we describe the signed graphs in $\SG$.
The disconnected ones, the signed complete graphs, and those with at most one eigenvalue unequal to $\pm 1$ 
were found in \cite{HT} (see Section~\ref{pre}) and for the ones switching isomorphic with an unsigned 
graph or its negative we refer to \cite{CHVW}.

\begin{thm}
Suppose $G^\sigma$ is a connected signed graph with two eigenvalues unequal to $\pm 1$, which is not a signed complete graph.
Then $G^\sigma$ or its negative $G^{-\sigma}$ is switching isomorphic with an unsigned graph in $\SG$, or with a signed graph represented by one of the following matrices.
\\[5pt]
$
A_1=\left[
\begin{array}{cc}
\!J\!-\!I_m\! & J \\ \!J & \!-R_{2\ell\!}
\end{array}
\right]
$
($m,\ell\geq 2$) 
with spectrum $\{-1^{\ell+m-2},1^\ell,\frac{1}{2}(m-2\pm\sqrt{m(m+8\ell)}\,)\}$.
\\[5pt]
$
A_2=\left[
\begin{array}{cc}
\!R_{2m}\! & J \\ J & \!-R_{2\ell}\!
\end{array}
\right]
$ 
($m,\ell\geq 2$) with spectrum $\{-1^{m+\ell-1},1^{m+\ell-1},\pm\sqrt{1+4m\ell}\}$.
\\[5pt]
$
A_3=\left[
\begin{array}{rrcr}
\!J\!-\!I_m\! & \1     & \1  & O\ \\
\ \1^\top\    & 0      &  1  & -\1^\top \\
\ \1^\top\    & 1      &  0  &  \1^\top \\
O\ \          & \!-\1  & \1  & \!I_\ell-\!J\!
\end{array}
\right]
$
($m,\ell\geq 1$) with spectrum $\{-1^m,1^\ell,-\ell-1,m+1\}$.
\\[5pt]
$
A_4=\left[
\begin{array}{cccc}
\!J\!-\!I_m\! & J                & \, J  & O \\
J\ \          & \!I_\ell\!-\!J\! & \, O  & J \\
J\ \          & O                & R_2\! & O \\
O\ \          & J                & \, O  & \!-R_2\! 
\end{array}
\right]
$\hspace{-3pt}  
\begin{tabular}{ll}
($m,\ell\geq 1$) with spectrum
\\[3pt]
$\{-1^{m+1},1^{\ell+1},
\frac{1}{2}(m-\ell\pm\sqrt{m^2+\ell^2+6m\ell+4m+4\ell+4})\}$.
\end{tabular}
\\[5pt]
$
A_5=\left[
\begin{array}{ccc}
\! J\! -\! I_m \! & J            & J \\
J           &\! J\!-\!I_\ell \!& O \\
J           & O            &\! I_k\!-\! J \! \\
\end{array}
\right]
$
\hspace{-3pt}
\begin{tabular}{ll}
($k \geq 2$, $(m,\ell)=(3,8)$, $(4,6)$, or $(6,5)$)
with spectra
\\[3pt]
$\{-1^{9},1^{k},\frac{1}{2}(9-k\pm\sqrt{k^2+26k+121})\}$,
\\[3pt]
$\{-1^{8},1^{k}, \frac{1}{2}(8-k
\pm\sqrt{k^2+28k+100})\}$,
\\[3pt]
$\{-1^{9},1^{k},\frac{1}{2}(9-k\pm\sqrt{k^2+38k+121})\}$.
\end{tabular}
\\[5pt]
$
A_6=\left[
\begin{array}{ccc}
\! J\!-\! I_m\! & J        & J \\
J     & \! I_\ell\!-\! J \! & O \\
J     & O        & \! R_{2k}\!
\end{array}
\right]
$\hspace{-3pt}
\begin{tabular}{ll}
($m\geq 1$, $(\ell,k)=(3,4)$, or $(4,3)$)
with spectra
\\[3pt]
$\{-1^{m+4},1^{5},\frac{1}{2}(m-1
\pm\sqrt{m^2+42m+9})\}$,
\\[3pt]
$\{-1^{m+3},1^{5},\frac{1}{2}(m-2
\pm\sqrt{m^2+40m+16})\}$.
\end{tabular}
\\[5pt]
$
A_7=\left[
\begin{array}{ccc}
R_{2m}& J        & J \\
J     & \! J\!-\! I_\ell\! & O \\
J     & O        & \! -R_{2k}\!
\end{array}
\right]
$\hspace{-3pt}
\begin{tabular}{ll}
($m\geq 1,~(\ell,k)=
(3,3)$, or $(4,2)$) 
with spectra
\\[3pt]
$\{-1^{m+4}, 1^{m+3},
\frac{1}{2}(1\pm\sqrt{8m+1}\}$,
\\[3pt]
$\{-1^{m+4}, 1^{m+2}, 1\pm 2\sqrt{4m+1}\}$.
\end{tabular}
\\[5pt]
$
A_{8}=\left[
\begin{array}{cccc}
R_2 & J & \1  & O \\
J   & \!I_m\!-\!J\! & \0  & J \\
\1^\to\! & \0^\to & 0 & \0^\to \\
O & J & \0 & -R_4 
\end{array}
\right]
$\hspace{-3pt}
\begin{tabular}{ll}
($m\geq 1$) with spectrum 
\\[3pt]
$\{-1^{3},1^{m+2},\frac{1}{2}(1-m\pm\sqrt{m^2+22m+9})\}$. 
\end{tabular}
\\[5pt]
$
A_{9}=\left[
\begin{array}{cccc}
R_{2m} & J & J & \0 \\
J   & R_{2} & O & \1 \\
J & O & -R_2 & \0 \\ 
\0^\to & \1^\to & \0^\to & 0 
\end{array}
\right]
$\hspace{-3pt}
\begin{tabular}{ll}
($m\geq 1$) with spectrum 
$\{-1^{m+2},1^{m+1},\frac{1}{2}(1\pm\sqrt{32m+9})\}$. 
\end{tabular}
\\[5pt]
$
A_{10}=\left[
\begin{array}{cccc}
\! J\!-\! I_m\!  & J & O & O \\
J & O  & J & \!J\!-\! I_\ell\! \\
O & J & \!I_m\!-\!J  & O \\ 
O & \! \! J\!-\!I_\ell\! & O & O 
\end{array}
\right]
$\hspace{-3pt}
\begin{tabular}{ll}
($(m,\ell)=(3,4)$, or $(4,3)$),
with spectra 
\\[3pt]
$\{-1^{6},1^{6},\pm 6\}$,
$\{-1^{6},1^{6},\pm 6\}$.
\end{tabular}
\\[5pt]
$A_{11}=\left[
\begin{array}{cccc}
\! J\!-\!I_m\!  & J & J & O \\
J & \! I_m\!-\!J\!  & O & J  \\
J & O & O & \! J\!-\!I_\ell\! \\
O & J & \! J\!-\!I_\ell\! & O
\end{array}
\right]
$\hspace{-3pt}
\begin{tabular}{ll}
($(m,\ell)=(3,4)$, or $(4,3)$)
with spectra
\\[3pt]
$\{-1^{6},1^{6},\pm 3\sqrt{5})\}$,
$\{-1^{6},1^{6},\pm 2\sqrt{13}\}$.
\end{tabular}
\\[5pt]
$
A_{12}=\left[
\begin{array}{cccc}
\! J\!-\! I_m\! & J & J & O \\
J & \! I_\ell\!-\!J\!     & O & J \\
J & O &   \! J\!-\!I_k\!    & J\\
O & J & J & \! I_j\!-\!J\!
\end{array}
\right]
$\hspace{-3pt}
\begin{tabular}{ll}
($(m,\ell,k,j)=(6,3,3,6)$, $(6,6,3,3)$, $(6,4,3,4)$, 
\\[3pt]
or
$(4,4,4,4)$),
with spectra 
$\{-1^{8},1^{8},\pm\sqrt{109}\}$,
\\[3pt]
$\{-1^{8},1^{8},\pm 10\}$,
$\{-1^{7},1^{8},\frac{1}{2}(-1\pm 3\sqrt{41})\}$,
\\[3pt]
$\{-1^{7},1^{7},\pm{9}\}$.
\end{tabular}
\\[5pt]
$
A_{13}=\left[
\begin{array}{cccc}
\! J\!-\! I_m\! & J              & O               & \0 \\
J               & \!-R_{2\ell}\! & J               & \1 \\
O               & J              & \! I_4\!-\! J\! & \0 \\ 
\0^\to          & \1^\to         & \0^\to          & 0 
\end{array}
\right]
$\hspace{-3pt}
\begin{tabular}{ll}
($(m,\ell)=(6,2)$, or $(5,3)$),
with spectra 
\\[3pt]
$\{-1^{7},1^{6},\frac{1}{2}(1\pm\sqrt{241})\}$,
$\{-1^{7},1^{7},\pm 6\sqrt{2}\}$.
\\[5pt]
\end{tabular}
\\[5pt]
$
A_{14}=\left[
\begin{array}{ccc}
\! J\!-\! I_m\! & J & O \\
J & \!-R_{2\ell}\! & J \\
O & J & \!-R_{4}\!
\end{array}
\right]
$
\hspace{-3pt}
\begin{tabular}{ll}
($(m,\ell)=(6,2)$, or $(5,3)$) 
with spectra
\\[3pt]
$\{-1^{7}, 1^{5}, 1\pm 2\sqrt{13}\}$,
$\{-1^7, 1^6, \frac{1}{2}(1\pm\sqrt{249}\}$.
\end{tabular}
\\[5pt]
$A_{15}=\left[
\begin{array}{cccc}
\! J\!-\!I_m\!  & J & J & O \\
J & \! I_\ell\!-\!J\! & O & J \\
J & O & \!R_{2k}\! & J \\
O & J & J & \! -R_{2j}
\end{array}
\right]
$\hspace{-3pt}
\begin{tabular}{ll}
($(m,\ell,k,j)=(3,3,3,3)$, $(4,3,3,2)$,
or $(4,4,2,2)$)
\\[3pt]
with spectra
$\{-1^{8},1^{8},\pm\sqrt{85}\}$,
$\{-1^{8},1^{7},\frac{1}{2}(1\pm\sqrt{313})\}$,
\\[3pt]
$\{-1^{7},1^{7},\pm\sqrt{73}\}$.
\end{tabular} 
\\[5pt]
$A_{16}=\left[
\begin{array}{cccc}
R_2 & J & J & O \\
J & -R_2 & O & J  \\
J & O & O & \!I_3\! \\
O & J & \!I_3\! & O
\end{array}
\right]
$\hspace{-3pt}
with spectrum $\{-1^{4},1^{4},\pm\sqrt{17}\}$. 
\\[5pt]
$
A_{17}=\left[
\begin{array}{cccccc}
R_2 & J & J & \1 & O &  \0 \\[-2pt]
J & -R_2 & O & \0 & J &  \1 \\[-2pt] 
J & O & R_2 & \0 &  J & \0 \\ 
\1^\to & \0^\to & \0^\to & 0 & \0^\to  & 0 \\[-2pt]  
O & J & J & \0 & -R_2 & \0  \\
\0^\to & \1^\to & \0^\to & 0 & \0^\to & 0 \\  
\end{array}
\right]
$
\hspace{-3pt}
with spectrum $\{-1^4, 1^4, \pm 2\sqrt{5}\}$.
\\[5pt]
$
A_{18}=\left[
\begin{array}{cccc}
\! J\!-\!I_m\! & J & \1 & O \\
J   & -R_{2\ell} & \0 & J \\
\1^\to & \0^\to & 0 & \0^\to \\ 
O & J & \0 & -R_2 
\end{array}
\right]
$\hspace{-3pt}
\begin{tabular}{ll}
($(m,\ell)=(4,3)$, or $(3,4)$) with spectra
\\[3pt]
$\{-1^{6},1^{5},\frac{1}{2}(1\pm\sqrt{177})\}$,
$\{-1^6,1^6,\pm3\sqrt{5}\}$. 
\end{tabular}
\\[5pt]
$A_{19}=\left[
\begin{array}{ccccc}
R_2 & J & J & \1 & O \\
J & \! I_m\!-\!J\! & O & \0 & J \\
J & O & \!R_{2\ell}\! & \0 & J \\
\1^\to & \0^\to & \0^\to & 0 & \0^\to \\
O & J & J & \0 & -R_2 \\
\end{array}
\right]
$\hspace{-3pt}
\begin{tabular}{ll}
($(m,\ell)=(3,3)$, or $(4,2)$)
with spectra
\\[3pt]
$\{-1^{6},1^{6},\pm 2\sqrt{10}\}$,
$\{-1^{5},1^{6},\frac{1}{2}(-1\pm 3\sqrt{17})\}$.
\end{tabular} 
\end{thm}
\begin{proof}
To see that each of the above matrices has the given spectrum, we use the same method as used in \cite{CHVW} and \cite{HT}.
For each matrix $A$ the presented block structure gives an equitable partition
(that is, each block has constant row and column sums).
Part of the spectrum of $A$ is the spectrum of the quotient matrix $Q$ (containing the row sums of the blocks), and it is straightforward to check that each $Q$ has just two eigenvalues different from $\pm 1$. 
The remaining eigenvalues of $A$ remain unchanged if we add or subtract an all-one block $J$ to some of the blocks.
One easily verifies that this can be done in such a way that we obtain a matrix with all eigenvalues equal to $\pm 1$.
The proof that the list is complete comprises the remainder of this paper.
\end{proof} 

Note that the above list is not free from overlap.
For example the negatives of $A_2$, $A_3$ and $A_4$
can alo be obtained by interchanging $m$ and $\ell$, and several cases with symmetric spectrum are switching isomorphic with their negatives.

Also note that $A_5$ with $k=1$ corresponds to the three unsigned graphs in $\SG$ given in Theorem~1$(v)$ of~\cite{CHVW}. 
Thus these three sporadic graphs from \cite{CHVW} are in fact part of an infinite family of signed graphs in $\SG$.

\section{Preliminaries}\label{pre}

We start with some results from \cite{HT}.

\begin{lem}\label{-1}
If a signed graph $G^\sigma$ has smallest eigenvalue at least $-1$, 
then the underlying graph $G$ is a disjoint union of complete graphs.
\end{lem}

The following two results are a bit more general than Proposition~2.3 and Theorem~2.4 in \cite{HT}.
The proofs, however, are basically the same.

\begin{prp}\label{but1}
If $G^\sigma$ is connected and all but at most one eigenvalues are in the interval $[-1,1]$,
then $G^\sigma$ or $G^{-\sigma}$ is switching isomorphic with an unsigned complete graph.
\end{prp}

\begin{thm}\label{disconnected}
If $G^\sigma$ is disconnected and has no isolated vertices or edges,
and all but two eigenvalues of $G^\sigma$ are in the interval $[-1,1]$,
then $G^\sigma$ is the disjoint union of two signed graphs each of which is switching isomorphic with an unsigned complete graph or its negative.
\end{thm}

\begin{thm}
If $G^\sigma\in\SG$ is bipartite, then $G^\sigma$ is switching isomorphic with an unsigned bipartite graph.
\end{thm}

The bipartite graphs in $\SG$ are described in \cite{vDS,CHVW,HT}.

\begin{thm}\label{complete}
Assume $K_n^\sigma$ is a signed complete graph of order $n$ for which the adjacency matrix
has all but two eigenvalues in the interval $[-1,1]$.
Then $K_n^\sigma$ is switching isomorphic with a signed graph with adjacency matrix
\[\J_{m,\ell} =
\left[
\begin{array}{cc}
\!\!\!J-I_m & J
\\
J &\! -J +I_\ell
\end{array}
\right]
\]
with $n=m+\ell$, $m,\ell\geq 2$
and spectrum
$\{-1^{m-1},1^{\ell-1},
\frac{1}{2}(m-\ell\pm\sqrt{m^2+\ell^2+6m\ell+4m+4\ell+4})\}$.
\end{thm}

\section{Restrictions}

In the remainder we will assume that $G^\sigma$ is a signed graph in $\SG$ of order $n$ 
which is connected, and not bipartite or complete.
Then, by Proposition~\ref{but1}, $G^\sigma$ has  two eigenvalues $r$ and $s$ different 
from $\pm 1$, and we may assume that $s<-1$ and $r>1$,
since otherwise $G^\sigma$ or its negative has smallest eigenvalue $-1$, so $G^\sigma$ 
would be disconnected or complete by Lemma~\ref{-1}.

An important tool is eigenvalue interlacing.
It follows that every induced signed subgraph of a graph $G^\sigma$ in $\SG$ has second largest eigenvalue at most 1,
and second smallest eigenvalue at least $-1$.
This excludes many subgraphs, including the path $P_n$ with $n\geq 6$ and the $n$-cycle with $n\geq 5$ (for every signing).
Some other forbidden induced signed subgraphs are presented in Figure~~\ref{forbidden}. 
\setlength{\unitlength}{1.5pt}
\begin{figure}[ht]
\begin{center}
\begin{picture}(30,30)(2,0)
\put(0,1){\circle*{3}}
\put(20,1){\circle*{3}}
\put(0,21){\circle*{3}}
\put(20,21){\circle*{3}}
\put(20,1){\line(0,1){20}}
\put(0,1){\line(0,1){20}}
\put(0,21){\line(1,0){20}}
\multiput(0,1)(4,0){5}{\line(1,0){2}}
\end{picture}
\begin{picture}(30,30)(-7,0)
\put(5,1){\circle*{3}}
\put(22,1){\circle*{3}}
\put(0,16.5){\circle*{3}}
\put(27,16.5){\circle*{3}}
\put(13.5,8){\circle*{3}}
\put(13.5,26.5){\circle*{3}}
\put(13.5,8){\line(0,1){19}}
\put(13.5,8){\line(4,-3){8}}
\put(13.5,8){\line(-4,-3){8}}
\put(5,1){\line(1,0){17}}
\put(13.5,8){\line(3,2){12}}
\put(13.5,8){\line(-3,2){12}}
\put(0,16.5){\line(1,0){27}}
\put(13.5,26.7){\line(4,-3){13.5}}
\put(13.5,26.7){\line(-4,-3){13.5}}
\end{picture}
\begin{picture}(30,30)(-18,0)
\put(5,1){\circle*{3}}
\put(22,1){\circle*{3}}
\put(0,16.5){\circle*{3}}
\put(27,16.5){\circle*{3}}
\put(13.5,26.5){\circle*{3}}
\multiput(5,1)(4,0){5}{\line(1,0){2}}
\put(5,1){\line(1,3){8.5}}
\put(5,1){\line(-1,3){5}}
\put(5,1){\line(4,3){22}}
\put(22,1){\line(-4,3){22}}
\put(0,16.5){\line(1,0){26.5}}
\put(22,1){\line(1,3){5}}
\put(13.5,26.7){\line(4,-3){13.5}}
\put(13.5,26.7){\line(-4,-3){13.5}}
\end{picture}
\caption{forbidden signed graphs; negative edges are dashed}\label{forbidden}
\end{center}
\end{figure}
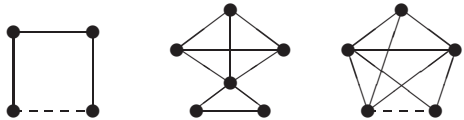
\setlength{\unitlength}{1.5pt}
\begin{figure}[ht]
\begin{center}
\begin{picture}(30,33)(20,0)
\put(5,1){\circle*{3}}
\put(22,1){\circle*{3}}
\put(0,16.5){\circle*{3}}
\put(27,16.5){\circle*{3}}
\put(13.5,26.5){\circle*{3}}
\put(5,1){\line(-1,3){5}}
\put(5,1){\line(1,0){17}}
\put(22,1){\line(1,3){5}}
\put(13.5,26.7){\line(4,-3){13.5}}
\put(13.5,26.7){\line(-4,-3){13.5}}
\end{picture}
\begin{picture}(30,33)(7,0)
\put(5,1){\circle*{3}}
\put(22,1){\circle*{3}}
\put(0,16.5){\circle*{3}}
\put(27,16.5){\circle*{3}}
\put(13.5,26.5){\circle*{3}}
\put(5,1){\line(1,0){17}}
\put(5,1){\line(-1,3){5}}
\put(0,16.5){\line(1,0){26.5}}
\put(22,1){\line(1,3){5}}
\put(13.5,26.7){\line(4,-3){13.5}}
\put(13.5,26.7){\line(-4,-3){13.5}}
\end{picture}
\begin{picture}(30,33)(-6,0)
\put(5,1){\circle*{3}}
\put(22,1){\circle*{3}}
\put(0,16.5){\circle*{3}}
\put(27,16.5){\circle*{3}}
\put(13.5,26.5){\circle*{3}}
\put(5,1){\line(4,3){22}}
\put(5,1){\line(-1,3){5}}
\put(0,16.5){\line(1,0){26.5}}
\put(22,1){\line(1,3){5}}
\put(13.5,26.7){\line(4,-3){13.5}}
\put(13.5,26.7){\line(-4,-3){13.5}}
\end{picture}
\begin{picture}(36,15)(-17,-3)
\put(0,10){\line(1,1){10}}
\put(0,10){\line(1,-1){10}}
\put(20,10){\line(-1,1){10}}
\put(20,10){\line(-1,-1){10}}
\put(20,10){\line(1,0){20}}
\put(10,0){\line(0,1){20}}
\put(0,10){\circle*{3}}
\put(20,10){\circle*{3}}
\put(40,10){\circle*{3}}
\put(10,0){\circle*{3}}
\put(10,20){\circle*{3}}
\put(30,10){\circle*{3}}
\end{picture}
\begin{picture}(36,15)(-37,-3)
\put(30,0){\line(0,1){20}}
\put(0,10){\line(1,1){10}}
\put(0,10){\line(1,-1){10}}
\put(10,20){\line(1,-1){20}}
\put(10,0){\line(1,1){20}}
\put(0,10){\circle*{3}}
\put(20,10){\circle*{3}}
\put(30,20){\circle*{3}}
\put(10,0){\circle*{3}}
\put(10,20){\circle*{3}}
\put(30,0){\circle*{3}}
\end{picture}
\caption{forbidden graphs for all signings}\label{forbidden-all}
\end{center}
\end{figure}
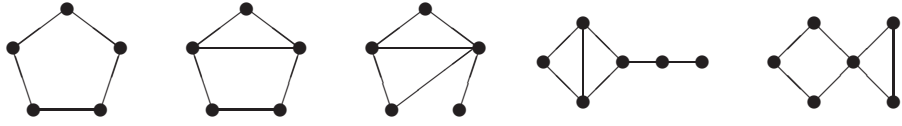

\noindent
Of course the signed graphs switching isomorphic with a graph in Figure~\ref{forbidden} and their negatives are also forbidden, 
but this does not include all possible signings.
For example, the unsigned $4$-cycle can occur.
Figure~\ref{forbidden-all} lists forbidden induced subgraphs for the underlying unsigned graph.
In other words, these graphs are forbidden for all signings.
Another useful fact is that if $A$ is the adjacency matrix of a signed graph in $\SG$,
then $A^2$ has eigenvalues $1$, $r^2$ and $s^2$, 
so $A^2-I$ is positive semi-definite (psd) and has rank 2.
The following Lemma is a  generalization of Lemma~2 from \cite{CHVW}.
The proof is the same.

\begin{lem}\label{deg1}
Suppose $G^\sigma\in\SG$ with adjacency matrix $A$, then
\\
$(i)$  if $G^\sigma$ has no isolated edges, then there is no vertex of degree $1$,
\\
$(ii)$ if $\x$ and $\y$ are two columns of $A$ such that $\x^\to\x = \x^\to\y$,
then $\x$ and $\y$ differ in at least three coordinate places, in particular no two columns (or rows) of $A$ are equal.
\end{lem}
\begin{proof}
$(i)$
Suppose $G^\sigma$ has a vertex $v$ of degree $1$ 
and a vertex $w$ at distance $2$ from $v$.
Then the $2\times 2$ submatrix of $A^2-I$ corresponding to $v$ and $w$ has a negative determinant which contradicts that $A^2-I$ is psd.
$(ii)$
In this case $A^2-I$ has a psd submatrix
\[ 
M=\left[
\begin{array}{cc}
\x^\to\x-1 & \x^\to\y \\ \y^\to\x & \y^\to\y-1
\end{array}
\right]
=
\left[
\begin{array}{cc}
\x^\to\x-1 & \x^\to\x \\ \x^\to\x & \y^\to\y-1
\end{array}
\right].
\]
So $\det(M)\geq 0$
which implies that $\y^\to\y > \x^\to\x+2$.
\end{proof}

\section{The possible formats}

Since $G^\sigma$ is not bipartite and contains no odd cycle of length 5 or more, $G^\sigma$ contains a triangle.

\begin{lem}\label{format}
Let $C$ be a maximal clique in $G^\sigma$ of size $c\geq 3$.
Then, up to vertex ordering, switching and taking the negative,
the $c\times n$ submatrix of $A$ corresponding to $C$ is one of the following:
\begin{enumerate}
\item[(i)]
$
 \left[
\begin{array}{c|c|c}
J-I_c
&
\begin{array}{c}
\phantom{-}J_{c_1\times x}\\-J_{c_2\times x}\\  O
\end{array}
&
O_{c\times z}
\end{array}
\right] \mbox{ with } c_1+c_2<c,\  c_1,c_2,x >0,\  z\geq 0,
$
\item[(ii)]
$
\left[
\begin{array}{c|c|c}
J-I_c
&
\begin{array}{cc}
J_{c_1\times x} & O
\\
O & J_{c_2\times y}
\end{array}
&
O_{c\times z}
\end{array}
\right] \mbox{ with } c_1+c_2=c,\  c_1,c_2,x >0,\  y,z \geq 0,
$
\item[(iii)]
$
\left[
\begin{array}{c|c|c}
\J_{c_1,c_2}
&
\begin{array}{cc}
J_{c_1\times x} & O
\\
O & J_{c_2\times y}
\end{array}
&
O_{c\times z}
\end{array}
\right]
$  with
$ c_1+c_2=c$,  $c_1,c_2\geq 2$, $x >0$, $y,z \geq 0$.
\end{enumerate}
\end{lem}

\begin{proof}
The clique $C$ is a signed complete graph which has, by interlacing, all but at most two eigenvalues in the interval $[-1,1]$.
By Lemma~\ref{-1} the adjacency matrix of $C$ is, up to switching and taking the negative, equal to  $J-I_c$, or $\J_{c_1,c_2}$,
where  $c_1+c_2=c$ and $c_1,c_2\geq 2$.
So the matrix $N$ consisting of the rows of $A$ corresponding to $C$ takes the form
$[\,J-I_c~|~L\,]$ or $[\,\J_{c_1,c_2}~|~\L\,]$.
Since $NN^\top-I$ is a principal submatrix of $A^2-I$,  the rank of $NN^\top-I$ is at most 2.
In case $N=[\,J-I_c~|~L\,]$ we get
\[
NN^\top -I = (c-2)J + LL^\top = MM^\top,
\]
where $M= [\,(\sqrt{c-2})\,\1~|~ L\,]$.
This implies that $M$, and hence $[\,\1~|~ L\,]$ has rank at most 2.
Let $\vv\neq\0$ be a column of $L$ ($L\neq O$, since $G^\sigma$ is connected).
Then $\vv$ has at least one entry equal to $0$ (because $C$ is maximal), and we may assume that $\vv$ has at least one positive entry
(otherwise we switch).
From rank$([\,\1~|~ L]\,) \leq 2$ it follows that every column of $L$ is equal to $\vv$, $\0$. or $\1-\vv$.
If $\vv$ has positive and negative entries, then $L$ has the form given in $(i)$, otherwise we get case $(ii)$.
In case $N=[\,\J_{c_1,c_2}~|~\L\,]$ we get
\[
NN^\top-I=
\left[\begin{array}{cc}
(c-2)J & (c_1-c_2)J_{c_1\times c_2}\\
(c_1-c_2)J_{c_2\times c_1} & (c-2)J
\end{array}
\right]+\L\L^\top = MM^\top,
\]
where
$M=[\begin{array}{c|c|c}
(\sqrt{c_1-1})\,\1 & (\sqrt{c_2-1})\,\s & \L
\end{array}]$, with $\s=[\,\1^\to_{c_1}\, |\, -\1^\to_{c_2}\,]^\to$.
This implies that $M$, and hence $[\,\1~|~\s~|~ \L\,]$ has rank at most 2.
Therefore every column of $\L$ is a linear combination of $\1$, and $\s$,
but has at least one zero entry.
So, up to switching, only the columns given in case $(iii)$ are allowed, and
since $G^\sigma$ is connected we may assume $x>0$.
\end{proof}

\section{Format $(i)$}

Assume $G^\sigma$ contains a clique $C$, which gives format $(i)$ from Lemma~\ref{format}.
We claim that $c_1=c_2=1$.
Indeed, if $c_1\geq 2$, or $c_2\geq 2$, then $G^\sigma$ has an induced
subgraph switching isomorphic with the last signed graph in Figure~\ref{forbidden}.
Next consider the matrix $N$ which consists of row $1$, row $3$ and row $c+i$ of $A$ ($1\leq i\leq x$).
Then
\[
N=
\left[
\begin{array}{crcc|c|c}
0 & 1 & 1 & \1^{\to}  & \1^\to_x & \0^\to_z \\[3pt]
1 &  1 & 0 &\1^\to  & \0^\to & \0^\to \\[3pt]
1 & -1 & 0 &\0^\to & \x^\to & \z^\to
\end{array}
\right],\ \mbox{and}
\]
\[
NN^\to-I=
\left[
\begin{array}{ccc}
c-2+x & c-2 & \x^\to\1-1\\
c-2 & c-2 & 0 \\
\x^\to\1-1 & 0 & 1+\x^\to\x+\z^\to\z
\end{array}
\right]
\]
for $(0,\pm 1)$-vectors $\x$ and $\z$.
Note that $\x$ has an entry on the diagonal of $A$, which is equal to $0$.
Therefore $|\x^\to\1| \leq \x^\to\x  \leq x -1$.
We have that $NN^\top-I$ is singular, so
\[
0=\det(NN^\to-I)=(c-2)(x(1+\x^\to\x+\z^\to\z)-(\x^\to\1-1)^2),
\]
which implies that $-\x^\to\1=\x^\to\x=x-1$ and $\z^\to\z=0$.
Finally, $\z=\0$ implies that $z=0$ (otherwise $G^\sigma$ is disconnected),
and thus we have:
{\begin{ccl}
If $G^\sigma$ has adjacency matrix $A$ with Format $(i)$ then $A=A_3$.
\end{ccl}

\section{Further restrictions}

In the remainder we will assume that $G^\sigma$ contains no maximal clique with Format~$(i)$,
and that the matrix $A$ of $G^\sigma$ takes the form of Format~$(ii)$ or $(iii)$ as described in Lemma~\ref{format}.
Concerning the maximal clique $C$ we will assume the following:

{\begin{asm}\label{asm}
~\\
$(i)$ The clique $C$ has maximum order $c$, 
\\
$(ii)$ $C$ has the largest number of outgoing edges among all cliques of order $c$.
\end{asm}

For both formats $(ii)$ and $(iii)$, $C$ is the union of two cliques $C_1$ and $C_2$,
such that the remaining vertices can be partitioned into three sets $X$, $Y$, and $Z$,
where the vertices in $X$ are adjacent to $C_1$ and not to $C_2$,
the vertices in $Y$ are adjacent to $C_2$ and not to $C_1$, 
and the vertices from $Z$ are adjacent to $C_1$ nor  $C_2$.
Note that the sets $X$, $Y$ and $Z$ may be empty,
but $X$ and $Y$ are not both empty (otherwise $G^\sigma$ is disconnected or complete) and we may choose that $X\neq\emptyset$
(otherwise we interchange $C_1$ with $C_2$ and $X$ with $Y$;
in case of format $(iii)$ we also need to switch and take the negative). 
Recall that $c_i=|C_i|~(i=1,2)$, $x=|X|$, $y=|Y|$ and $z=|Z|$.
For a subset $V$ of the vertex set of $G^\sigma$ we let 
$G^{\sigma}_V$ denote the subgraph of $G^\sigma$ induced by $V$.

\begin{prp}\label{restrictions}
~\\
$(i)$ $G^{\sigma}_X$ and $G^{\sigma}_Y$ are disjoint unions of cliques.
\\
$(ii)$ Two adjacent vertices in $G^\sigma_X$ or $G^\sigma_Y$ correspond to identical rows and columns
in $A+I$ or $A-I$.
\\ 
$(iii)$
If $c_{1}\geq 3$ [$c_2\geq 3$], 
then  all edges of $G^{\sigma}_Y$ [$G^{\sigma}_X$] have the same sign which is opposite to the sign of $C_{1}$ [$C_2$].
\\
$(iv)$ 
Every clique in $G^{\sigma}_Y$ [$G^{\sigma}_X$] has order at most $2$,
or one clique has order at least $3$ and all other cliques have order $1$ and all edges have the same sign opposite to the sign of $C_1$ [$C_2$].
\\
$(v)$ There are no negative edges between $X$ and $Y$.
\\
$(vi)$ A vertex from $Z$ cannot be adjacent to both $X$ and $Y$.

\end{prp}

\begin{proof}
$(i)$ Suppose $G^{\sigma}_X$ contains an induced path
$P_3$ then the subgraph of $G^\sigma$ induced by this path together with one vertex of $C_1$ and one vertex of $C_2$ is a forbidden subgraph (no.~3 in Figure~\ref{forbidden-all}). 
Therefore $G^{\sigma}_X$ contains no $P_3$, 
so $G^{\sigma}_X$ is a clique.
$(ii)$ 
Let $C_3$ be the clique in $G^\sigma_A$ containing the two adjacent vertices.
Then $(ii)$ follows from Lemma~\ref{format} applied to to the maximal clique $C_1\cup C_3$.
$(iii)$ 
Take an edge $e$ in $G^{\sigma}_X$, 
and consider the subgraph $H^\sigma$ of $G^{\sigma}_X$ induced by $e$, 
three vertices of $C_{2}$ and one vertex of $C_1$.
If $e$ has the same sign as the edges of $C_{2}$, 
then $H^\sigma$ is switching isomorphic with the second forbidden subgraph of Figure~\ref{forbidden}. 
This proves $(iii)$.
$(iv)$
Suppose $G^{\sigma}_Y$ contains a clique $C_3$ of order at least 3.
Then $c_1\geq 3$ and all edges of $C_3$ have the same sign.
If $G^{\sigma}_Y$ contains another clique of order at least 2, then $G^\sigma$ contains a signed graph switching isomorphic with the second forbidden signed subgraph of Figure~\ref{forbidden}.
This proves~ $(iv)$.
$(v)$ 
A negative edge between $X$ and $Y$ will create a $4$-cycle with one negative edge
which is forbidden.
$(vi)$
A vertex in $Z$ adjacent to a vertex of $X$ and a vertex from $Y$ creates an induced pentagon, or a pentagon with one chord, which are forbidden for all signings.
\end{proof} 
Note that Assumption~\ref{asm} implies that a clique in $G^{\sigma}_X$ [or $ G^{\sigma}_Y$] has order at most $c_{2}$ [$c_1$], and for a clique $C_3$ in $G^\sigma_X$ [$G^\sigma_Y$] of order $c_2$ [$c_1$] there are at most $c_2 y$ [$c_1 x$] edges between $C_3$ and $Y\cup Z$ [$X\cup Z$].

\section{Format~$(ii)$ with $Y=\emptyset$}

Suppose $G^\sigma$ has format $(ii)$ and $Y=\emptyset$.
Let $N$ be the matrix consisting of three rows of $A$
corresponding to a vertex from $C_1$, a vertex from $C_2$
and a vertex $v$ from $X$.
Then
\[
N=
\left[
\begin{array}{cl|cl|c|c}
0 & \1^\to_{c_1-1} & 1 & \1^\to_{c_2-1} & \1^\to_{x} & \0^\to_z \\[3pt]
1 & \1^\to         & 0 & \1^\to         & \0^\to     &  \0^\to  \\[3pt]
1 & \1^\to         & 0 & \0^\to         & \x^\to_1   &  \z^\to_1
\end{array}
\right]
\]
(for $(0,\pm 1)$ vectors $\x_1$ and $\z_1$), and
\[
M=NN^\top-I=
\left[
\begin{array}{ccc}
c+x-2    & c-2 & c_1+x_1' -1 \\[3pt]
c-2      & c-2 & c_1       \\[3pt]
c_1+x_1'-1 & c_1 & c_1+x_1+z_1-1
\end{array}
\right],
\]
where $x_1'=\1^\to\x$, $x_1=\x^\to_1\x_1$ and $z_1=\z_1^\to\z_1$.
Note that $|x_1'|\leq x_1 \leq x-1$, $x_1+1\leq c_2$ and $x_1+1=c_2$ implies that $z_1=0$
by Assumption~\ref{asm}.
Using Gaussian elimination we get 
$\det(M)= x(c-2)(c_1+x_1+z_1-1)-(c-2)(x_1'-1)^2-xc_1^2$.
For convenience we define $f(a,b)=ab/(a+b)$.
We have $\det(M)=0$, which gives:
\begin{equation}\label{det1}
x_1+z_1+f(c_1,c_2-2) = \frac{(x_1'-1)^2}{x}+1.
\end{equation}
If $c_2=1$,
Proposition~\ref{restrictions}($i-iii$) implies that $X$ is a coclique and that there are no edges between $X$ and $Z$.
Then $A$ has $x+1$ equal rows, which contradicts Lemma~\ref{deg1}$(ii)$.
Therefore $c_2\geq 2$ and $f(c_1,c_2-2)\geq 0$.
From Assumption~\ref{asm} it follows that $x_1=x_1'=1$, or $x_1=-x_1'$.  
We will consider the different possibilities for $x_1$.

\begin{prp}\label{0}
Suppose $x_1=0$. 
Then $v$ is the only isolated vertex of $G^\sigma_X$ and one of the following holds: 
\\
$(i)$ $x=1$, $c_2=2$ and $A=A_4$ with $m=1$ and $\ell\geq 1$. 
\\
$(ii)$ $A$ is an unsigned matrix given in case $(v)$ of Theorem~1 from \cite{CHVW}.
\\
$(iii)$ $(c_1,c_2,x)\in\{(2,5,5),(2,6,3),(3,4,5),(4,4,3)\}$.
\end{prp}

\begin{proof}
Equation \ref{det1} gives $z_1=1+1/x-f(c_1,c_2-2)$, so $z_1\leq 2$.
If $z_1=2$ then $x=1$ and $c_2=2$, 
which straightforwardly leads to matrix $A_4$ with $m=1$. 
If $z_1=1$ then $f(c_1,c_2-2)=1/x$, which implies $(c_1,c_2,x)=(2,4,1)$, or $(1,3,2)$.
Both of these cases cannot be completed to a signed graph in $\SG$.
If $z_1=0$ then $f(c_1,c_2-2)=1+1/x$, 
which implies that $(c_1,c_2,x)$ is equal to $(4,2,1)$, $(3,8,1)$, $(6,5,1)$, 
$(3,5,2)$, $(4,4,3)$, $(2,6,3)$, $(2,5,5)$, or $(3,4,5)$.
The first three triples lead to case $(v)$ of Theorem~1 from~\cite{CHVW}.
If $G^\sigma_X$ has a second isolated vertex then the two corresponding rows of $A$ are equal which contradicts Lemma~\ref{deg1}. 
Therefore $G^\sigma_X$ has just one isolated vertex and the triple $(3,5,2)$ is not possible.
\end{proof}

\begin{prp}
If $x_1 = x-1 \geq 2$ then 
$A=A_5$ with $k\geq 3$.
\end{prp}
\begin{proof}
The underlying graph $G_X$ of $G^\sigma_X$
is complete, and from Proposition~\ref{restrictions}($iii$) it follows that all signs are negative.
Formula~\ref{det1} becomes: $z_1 = 2 - f(c_1,c_2-2)$.
If $z_1=2$, then $c_2=2$ and $x_1\leq1$,
which is excluded.
If $z_1=1$ then $f(c_1,c_2-2)=1$, so $c_1=2$, $c_2=4$ and $x=2$ or $3$ which has no solution by straightforward verification.
If $z_1=0$, then $f(c_1,c_2-2)=2$, which gives $(c_1,c_2)=(3,8)$, $(4,6)$ or $6,5)$. 
Thus we find $A_5$ from Theorem~\ref{main} (actually, only the ones with $k \leq \ell$,
because of Assumption~\ref{asm}).
\end{proof}

\begin{prp}
If $x_1=1$, then one of the following holds:  
\\ 
($i$) $A=A_5$ with $k=2$.
\\
($ii$) $A=-A_6$ with $m=1$ (switched around the first vertex).
\\
($iii$) $A=A_7$.
\\
($iv$) $A$ is an unsigned matrix given in case $(ii)$ of Theorem~1 from \cite{CHVW}.
\end{prp}
\begin{proof}
If $x_1=-x_1'=1$, then clearly $x\geq 2$ and Formula~\ref{det1} becomes 
$z_1=4/x - f(c_1,c_2-2))$.
It also follows that $c_2\geq 3$, since
$c_2=2$ would imply $z_1=0$ (by Assumption~\ref{asm}). 
From $x\geq 2$ and $f(c_1,c_2-2)\geq 1$ it follows that $z_1\leq 1$, and that $z_1=1$ implies $(c_1,c_2,x)=(2,4,2)$,
which is not possible. 
So $z_1=0$ and we obtain the 
following solutions for Formula~\ref{det1}: 
$(c_1,c_2,x)=(1,3,8)$, $(1,4,6)$, $(1,6,5)$, $(2,3,6)$, $(2,4,4)$,
$(2,6,3)$, $(3,8,2)$, $(4,3,5)$, $(4,4,3)$, $(4,6,2)$, or $(6,5,2)$.
By Proposition~\ref{restrictions}($iv$),
$G^\sigma_X$ has an isolated vertex if $x$ is odd, in which case we need a solution with $x_1=0$ for the same $(c_1,c_2,x)$.
Proposition~\ref{0} gives just two possibilities: $(2,6,3)$ and $(4,4,3)$.
In both cases $G^\sigma_X$ consists of one negative edge and an isolated vertex, and it is easily checked that this is not possible.
Thus $x$ is even, and of the remaining seven triples the first two lead to $-A_6$ with $m=1$, the next two to $A_7$, and the last three to $A_5$ with $k=2$.
If $x_1=x_1'=1$, then we easily have $c_2=2$ and $z_1=0$.
None of the solutions above has $c_2=2$ and $x\geq 2$, 
therefore all rows corresponding to $X$ have $x_1=x_1'=1$, $c_2=2$ and $z_1=0$,
which implies that $G^\sigma$ is case~$(ii)$ 
of Theorem~1 from~\cite{CHVW}.
\end{proof}

\begin{prp}
If $2\leq x_1 \leq x-2$, then there is no solution.
\end{prp}
\begin{proof}
Proposition~\ref{restrictions}$(iv)$ implies that if there is a solution with $2\leq x_1 \leq x-2$, then there is also one with $x_1=0$ for the same $c_1$, $c_2$ and $x$.
From Proposition~\ref{0} we find that there is just one candidate:
$(c_1,c_2,x)=(3,4,5)$ and $x_1=-x_1'=2$.
But this does not satisfy Formula~\ref{det1}.
\end{proof}
Finally we revisit Proposition~\ref{0}($iii$).
For each of the four cases there must be another solution with $x_1>0$, which we haven't found. So none of these four triples are possible, and the determination all graphs with Format ($ii$) and $Y=\emptyset$ is complete.

{\begin{ccl}
If $A$ has Format $(ii)$ with $Y=\emptyset$ then $A$ is unsigned (case~$(ii)$, or $(v)$   from Theorem~1 of \cite{CHVW}), or $A$ is one of the following:
$(i)$ $A_4$ with $m=1,\ell\geq 1$, 
$(ii)$ $A_5$ with $2\leq k\leq \ell$,
$(iii)$ $-A_6$ with $m=1$ (switched around the first vertex), 
$(iv)$ $A_7$.
\end{ccl}

\section{Format~$(iii)$ with $Y=\emptyset$}

Suppose $G^\sigma$ has format $(iii)$ and $Y=\emptyset$.
Again we let $N$ be the matrix consisting of row~1, row~$c_1+1$ and row $c+1$ of $A$
(so row~$c+1$ corresponds to a vertex in $X$).
Then
\[
M=NN^\top-I=
\left[
\begin{array}{ccc}
c_1+c_2+x-2 & c_1-c_2 & c_1+x_1'-1     \\[3pt]
c_1-c_2     & c_1+c_2 & c_1          \\[3pt]
c_1+x_1'-1    & c_1     & c_1+x_1+z_1-1
\end{array}
\right],
\]
with $x_1'$, $x_1$ and $z_1$ as before
(so $|x_1'|\leq x_1 \leq x-1$).
The expression for $\det(M)$ is a bit messy now, so to deal with $\det(M)=0$ in another way.
We write $M=(c_1-2)J+(c_2-2)K+R$, where 
 \[
 K=\left[\begin{array}{ccc} 1 & -1 & 0 \\ -1 & 1 & 0 \\ 0 & 0 & 0 \end{array}\right],
 \mbox{ and }
 R=\left[\begin{array}{ccc} x+2 & 0 & x_1'+1 \\ 0 & 2 & 2 \\ x_1'+1 & 2 & x_1+z_1+1 \end{array}\right].
 \]
We easily have that $(c_1-2)J$ and $(c_2-2)K$ are positive semi-definite (psd) matrices.
Thus if $R$ is positive definite (pd), then $M$ is pd and therefore nonsingular.
Moreover, if $R$ is psd, then $M$ is psd and the kernel of $M$ is the intersection of the kernels of $(c_1-2)J$, $(c_2-2)K$ and $R$.
So if this intersection is $\{\0\}$ then $M$ is also nonsingular.
The diagonal entries of $R$ are positive, and so are the $2\times 2$ principal minors.
This implies that $R$ is pd whenever $\det(R)>0$ and $R$ is psd when $\det(R)=0$.
Using this, and 
$\det(R) = 2((x+2)(x_1+z_1-1)-(x_1'+1)^2)$ we get the following restrictions.

\begin{prp}\label{iii-0-p}
One of the following holds:
\\
($i$) $x_1=x_1'$, $z_1=0$, 
\\
($ii$) $x_1=x_1'=0$, $z_1 =1$,
\\
($iii$) $x_1=-x_1'=1$, $z_1=0$, $c_2=2$.
\end{prp}
\begin{proof}
From Proposition~\ref{restrictions}($iii,iv$) it follows that $x_1=-x_1'=1$, or $x_1=x_1'$.
First assume $x_1=x_1'$.
Using $x\geq x_1+1$ we get 
$\det(R)\geq 2(x_1 z_1+3z_1-4)$.
This implies $z_1=0$, $(x_1,z_1)=(0,1)$, or $(1,1)$.
The first option is case $(i)$, and the second one is case $(ii)$. 
If $x_1=z_1=1$ we find 
\[ 
R=\left[\begin{array}{ccc} x+2 & 0 & 2 \\ 0 & 2 & 2 \\ 2 & 2 & 3 \end{array}\right].
\]
We have $x\geq x_1+1=2$ and therefore $R$ is pd if $x>2$ and psd if $x=2$.
If $x=2$ the kernel of $R$ is spanned by $[1, 2, -2]^\top$, which is not in the kernel of $J$ or $K$.
Therefore $M$ is singular only if $c_1=c_2=2$ and $x=x_1+1=2$.
Thus we have $c_2=x_1+1$ and $z_1>0$, which contradicts Assumption~\ref{asm}.
Next we assume $x_1=-x_1'=1$.
Then $\det(R)=2(x+2)z_1$, so $z_1=0$ and 
\[
R=\left[\begin{array}{ccc} x+2 & 0 & 0 
\\ 0 & 2 & 2 \\ 0 & 2 & 2 \end{array}\right].
\]
Again $R$ is pd if $x>2$ and psd if $x=2$.
If $x=2$ the kernel of $R$ is spanned by $[0,1,-1]^\top$, which is not in the kernel of $K$. 
Therefore $c_2=2$ and we have case $(iii)$.
\end{proof}
We claim that $Z=\emptyset$.
Suppose not, and assume $u\in Z$, then $u$ cannot be adjacent to two (or more) vertices of $X$, because that would mean that $A$ has two equals rows which is not possible.
But $u$ has degree at least 2, so there is another vertex $v$ in $Z$ adjacent to $u$.
But then we create a forbidden subgraph
from Figure~\ref{forbidden-all}.
This implies that case~(ii) of the above proposition does not occur.
Note that $G^\sigma_X$ has at most one isolated vertex, since otherwise $A$ has two equal rows. 
Also $G^\sigma_X$ cannot contain a positive isolated edge if $c_2=2$. Indeed, 
interchanging $C_2$ with the positive edge gives a signed graph with Format~($ii$) and $y=0$, 
which has been dealt with in the previous section.
Thus Proposition~\ref{iii-0-p}
leads to just six possible structures for  
$G^\sigma_X$, being:
\\
($i$) $G^\sigma_X$ is a clique of order $x$ with all edges positive,
\\
($ii$) $G^\sigma_X$ is a clique of order $x-1$ with all edges positive extended with an isolated vertex,
\\
($iii$) $c_2>2$ and $G^\sigma_X$ is the disjoint union of $x/2$ positive edges,
\\
($iv$) $c_2>2$ and $G^\sigma_X$ is the disjoint union of $(x-1)/2$ positive edges and one isolated vertex,
\\
($v$) $c_2=2$ and $G^\sigma_X$ is the disjoint union of $x/2$ negative edges,
\\
($vi$) $c_2=2$ and $G^\sigma_X$ is the disjoint union of $(x-1)/2$ negative edges and one isolated vertex.  
\\
For each of these cases the matrix $A$ of $G^\sigma$ has an equitable 
partition with quotient matrices
\[
Q_1=\left[\begin{array}{ccc} 
\!c_1\!-\!1\! & c_2 & x \\ 
c_1 & \!1\!-\!c_2\! & 0 \\ 
c_1 & 0 & \!x\!-\!1\! 
\end{array}\right],\,
Q_2=\left[\begin{array}{cccc} 
\!c_1\!-\!1\! & c_2 & \!x\!-\!1\! & 1 \\ 
c_1 & \!1\!-\!c_2\! & 0 & 0\\ 
c_1 & 0 & \!x\!-\!2\! &0\\
c_1 & 0 & 0 & 0 
\end{array}\right],\,
Q_3=\left[\begin{array}{ccc} 
\!c_1\!-\!1\! & c_2 & x \\ 
c_1 & \!1\!-\!c_2\! & 0 \\ 
c_1 & 0 & 1
\end{array}\right],
\]
\[
Q_4=\left[\begin{array}{cccc} 
\!c_1\!-\!1\! & c_2 & \!x\!-\!1\! & 1 \\ 
c_1 & \!1\!-\!c_2\! & 0 & 0\\ 
c_1 & 0 & 1 &0\\
c_1 & 0 & 0 & 0 
\end{array}\right],\,
Q_5=\left[\begin{array}{cc} 
\!c_1\!-\!1\! & \!x\!+\!2\! \\ 
c_1 & -1 
\end{array}\right],\,
Q_6=\left[\begin{array}{ccc} 
\!c_1\!-\!1\! & \!x\!+\!1\! & 1 \\ 
c_1 & -1 & 0 \\ 
c_1 & 0 & 0
\end{array}\right],
\]
respectively.
Each of these quotient matrices must have at most two eigenvalues different from $\pm 1$. 
For $Q_5$ this is obvious and we find matrix $A_1$ of Theorem~\ref{main}. 
For the other five quotient matrices we checked for eigenvalues $\pm 1$ (simply by computing rank$(Q_i+I)$ and rank$(Q_i-I)$).
For $Q_2$ and $Q_4$ there is at most one eigenvalue equal to $\pm 1$,
and $Q_6$ has no eigenvalue equal to $\pm 1$.
However $Q_1$ has an eigenvalue $\pm 1$ whenever  $(c_1,x)=(3,8), (4,6)$, or $(6,5)$, 
which gives $A_5$ of Theorem~\ref{main} when $k>\ell$.
The quotient matrix $Q_3$ has eigenvalue $\pm 1$ whenever $(c_2,x)=(3,8), (4,6)$, or $(6,5)$, but $x$ is even so only the first two cases remain, and we find $A_6$ of Theorem~\ref{main} (when $m>1$).

{\begin{ccl}
If $G^\sigma$ has adjacency matrix $A$ with Format $(iii)$ with $Y=\emptyset$, then 
$A$ is one of the matrices represented by
$A_1$, $A_5$ (with $k >\ell$), or $A_6$ (with $m>1$). 
\end{ccl}

\section{$X$, $Y$ and $Z$ nonempty}

Take $v\in Z$.
By Proposition~\ref{restrictions}$(vi)$ we may assume that $v$ is adjacent to a vertex $w\in X$ and not adjacent to $Y$.
Consider the matrix $N$ consisting of four rows of $A$ corresponding to a vertex from $C_1$, a vertex from $C_2$, to $w$ and $v$, respectively.
Then
\[
N=
\left[
\begin{array}{clcllll}
0 & \1^\to_{c_1-1} & 0 & ~\1^\to_{c_2-1} & \1^\to_{x} & \0^\to_{y} & \0^\to_z \\[3pt]
1 & \1^\to & 0 & \epsilon\1^\to & \0^\to & \1^\to & \0^\to \\[3pt]
1 & \1^\to & 0 & ~\0^\to & \x_1^\to & \y_1^\to & \z_1^\to \\[3pt]
0 & \0^\to & 0 & ~\0^\to & \x_3^\to & \0^\to   & \z_3^\to
\end{array}
\right]
\]
for $(0,\pm 1)$~vectors $\x_1$, $\x_3$, $\y_1$ $\z_1$, $\z_3$, and $\epsilon=\pm1$, 
where $\epsilon=1$ in case of format~$(ii)$ and $\epsilon=-1$ for format~$(iii)$. 
Note that we may switch such that $\z_1\geq\0$ and that $\y_1\geq\0$ by Proposition~\ref{restrictions}$(v)$.
Define $M=NN^\to-I$, then
\[
M=
\left[
\begin{array}{cccc}
c_1+c_2+x-2  & c_1-1+\epsilon(c_2-1) & c_1+x_1' -1 & x_3' \\[3pt]
c_1-1+\epsilon(c_2-1) & c_1+c_2+y-2 & c_1+y_1 & 0  \\[3pt]
c_1+x_1'-1 & c_1+y_1 & c_1+x_1+y_1+z_1-1 & x_{1,3}+z_{1,3} \\[3pt]
x_3' & 0 & x_{1,3}+z_{1,3} & x_3+z_3-1
\end{array}
\right],
\]
where $x_1=\x_1^\to\x_1$, $x_1'=\x_1^\to\1$,  $x_{1,3}=\x_1^\to\x_3$, $x_3=\x_3^\to\x_3\geq 1$,
$x_3' =\x_3^\to\1\neq 0$,
$y_1=\y_1^\to\y_1=\y_1^\to\1$,
$z_1=\z_1^\to\z_1\geq 1$, $z_3=\z_3^\to\z_3$ and $z_{1,3}=\z_1^\to\z_3$.
Since $M$ is a principal submatrix of $A^2-I$, we know rank$(M)\leq 2$.
Define $M_{i,j}$ to be the matrix obtained from $M$ by deletion of row~$i$ and column $j$.
Then $M_{i,j}$ is singular for $1\leq i,j \leq 4$.
We treat the two considered formats seperately, and start with $(iii)$.

\subsection{Format $(iii)$}
In this case $\epsilon=-1$, $c_1,c_2\geq 2$ and we write
$M_{3,3} = (c_1-2)K +(c_2-2)L + R$, where 
\[ K=\left[\begin{array}{ccc}
1 & 1 & 0\\1 & 1 & 0\\0 & 0 & 0
\end{array}\right],~~
L=\left[\begin{array}{rrr}
1 & -1 & 0\\-1 & 1 & 0\\0 & 0 & 0
\end{array}\right],~~
R=\left[\begin{array}{ccc}
x+2 & 0 & x_3'\\0 & y+2 & 0\\x_3' & 0 & x_3+z_3-1
\end{array}\right].
\]
Clearly $K$ and $L$ are psd, so if $R$ is pd, then $M_{3,3}$ is non-singular. 
Therefore $R$ is not pd and 
$0\geq \det(R) = (y+2)((x+2)(x_3+z_3-1)-(x_3')^2)$.
This implies $z_3=0$, $x_3=|x_3'|=x$ and $x_3\leq 2$. 
We know $x_3>0$, and if $x_3=1$ then $x_3'\neq 0$ and 
$\det(M_{3,3})=-(y+2)(x_3')^2\neq 0$.
Thus $x_3=2$.
Then $\det(R)=0$ and the kernel of $R$ is spanned by
$[1, 0, -2]^\top$, which is not in the kernel of $K$ and $L$, and therefore $M_{3,3}$ is singular only if $c_1=c_2=2$.
Thus $x=x_3=c_1=c_2=2$, $z_3=0$ and
\[
M_{1,1}=(y_1+2)K+
\left[\begin{array}{ccc}
y-y_1 & 0 & 0\\0 & x_1+z_1-1 & x_1'\\0 & x_1' & 1
\end{array}\right].
\]
One easily verifies that both terms in this formula are psd
(recall that $|x_1|\leq x_1\leq 1$ and $z_1\geq 1$), and that the kernel of the sum is only nontrivial if $y=y_1$ and $z_1=1$.
If $x_1=0$, then the two rows of $A$ corresponding to $X$ are equal, which contradicts Lemma~\ref{deg1}($ii$), and $x_1=1$ is ruled out by Proposition~\ref{restrictions}($iii$).
So we conclude that $Z=\emptyset$ in case of format~$(iii)$.
(But notice that without Assumption~\ref{asm} we would have found $A_{17}$.)  

\subsection{Format~($ii$)}
 
In this case $\epsilon=1$, $c=c_1+c_2\geq 3$ and $M_{3,3}$ can be written as: $M_{3,3} = (c-2)K + R$, with $K$ as before and
\[ R=
\left[\begin{array}{ccc}
x & 0 & x_3'\\0 & y+1 & 0\\x_3' & 0 & x_3+z_3-1
\end{array}\right].
\]
With similar arguments as above we see that $R$, and hence $M_{3,3}$ is pd if $z_3>0$, if $|x_3'|<x_3$ and if $x_3<x-1$.
Therefore $z_3=0$, $|x_3'|=x_3$ and $x_3\in\{x,x-1\}$.

Assume that $x_3=x-1$.
Then using Gaussian elimination we easily obtain that $y(c-2)(x-2)-(y+c-2)=\det(M_{3,3})=0$. 
Hence $x-2=1/(c-2)+1/y$,
which leads to $(c,x,y)=(3,4,1)$, or $(4,3,2)$.
In both cases there are just a few possibilities for $A$, and it follows by straightforward verification that
there are no solutions (but without Assumption~\ref{asm} we would have found adjacency matrices $A_{8}$ with $m=1$, and $A_{17}$).

Next assume that $x_3=x$.
Then 
\[ M_{3,3}=
\left[\begin{array}{ccc}
x+c-2 & c-2 & x\\c-2 & y+c-2 & 0\\x & 0 & x-1
\end{array}\right].
\]
and $\det(M_{3,3})=y(c-2)(x-1)-x(y+c-2)=0$,
which implies $c-3=(x+y)/(xy-x-y)$.
The latter equations has only ten feasible solutions, being: 
$(c,x,y)=(4,4,4)$, $(4,3,6)$, $(4,6,3)$, $(5,3,3)$, $(5,2,6)$, $(5,6,2)$, $(6,2,4)$, $(6,4,2)$, $(8,2,3)$, $(8,3,2)$.
Observe that if $Z$ has a second vertex  $u$ then it cannot be adjacent to all vertices in $X$ (otherwise $A$ would have two indentaical rows), so $u$ is adjacent to all vertices of $Y$ and none of $X$.
This shows that $z\leq 2$.
Using this observation we find that for each of the above triples there are only a few possibilities for $A$, and we check them case by case.
The cases $(5,2,6)$ and $(4,3,6)$ lead to matrix $A_{10}$,
and the cases $(5,3,3)$ and $(4,4,4)$ lead to case $(vi)$ of Theorem~1 from \cite{CHVW}.
The other six cases have no solution.
Thus we conclude:

\begin{ccl}
If $A$ has format $(ii)$ or $(iii)$ with $X$, $Y$ and $Z$ nonempty, then $A=A_{10}$, or $A$ is the unsigned case~$(vi)$ of Theorem~1 from \cite{CHVW}.
\end{ccl}

\section{Format(iii) with $X$, $Y$ nonempty and $Z$ empty}

Throughout this section $G^\sigma$ has adjacency matrix $A$ with format~$(iii)$, $X, Y\neq\emptyset$ and $Z=\emptyset$.
Consider the matrix $N$ consisting of four rows of $A$ corresponding to a vertex from $C_1$, 
a vertex from $C_2$, a vertex $v$ from $X$ and a vertex from $w$ from $Y$.
Then
\[
N=
\left[
\begin{array}{clclll}
0 & \1^\to_{c_1-1} & 0 & ~\1^\to_{c_2-1} & \1^\to_{x} & \0^\to_{y} \\[3pt]
1 & \1^\to         & 0 & -\1^\to         & \0^\to     & \1^\to     \\[3pt]
1 & \1^\to         & 0 & ~\0^\to         & \x_1^\to   & \y_1^\to   \\[3pt]
0 & \0^\to         & 1 & ~\1^\to         & \x_2^\to   & \y_2^\to   
\end{array}
\right]
\]
for $(0,\pm 1)$~vectors $\x_1$, $\y_2$, and $(0,1)$~vectors $\y_1$, $\x_2$.
Define $M=NN^\to-I$, then
\[
M=
\left[
\begin{array}{cccc}
c_1+c_2+x-2  & c_1-c_2 & c_1+x_1'-1 & c_2+x_2 \\[3pt]
c_1-c_2 & c_1+c_2+y-2 & c_1+y_1 & 1-c_2+y'_2 \\[3pt]
c_1+x_1'-1 & c_1+y_1 & c_1+x_1+y_1-1 & x_{1,2}+y_{1,2} \\[3pt]
c_2+x_2 & 1-c_2+y'_2 & x_{1,2}+y_{1,2} & c_2+x_2+y_2-1
\end{array}
\right],
\]
where $x_1=\x_1^\to\x_1$,
$x_1'=\x_1^\to\1$,  
$x_{1,2}=\x_1^\to\x_2$, 
$x_2=\x_2^\to\x_2=\x_2^\to\1$,
$y_1=\y_1^\to\y_1=\y_1^\to\1$,
$y_2=\y_2^\to\y_2$,
$y'_2=\y_2^\to\1$,  
and $y_{1,2}=\y_1^\to\y_2$.
As before we define $M_{i,j}$ to be the matrix obtained from $M$ by deletion of row~$i$ and column $j$.
We have rank$(M)\leq 2$, hence $\det(M_{i,j}) = 0$ for $1\leq i,j \leq 4$.

\begin{prp}
Suppose $y-y_1\geq 2$ for some vertex in $X$, then one of the following holds.
\\
$(i)$: $x_1=x_1'=1$, $y_1=0$, $x=y=2$ which leads to matrix $A_4$ with $m,\ell\geq 2$.
\\
$(ii)$: $x_1=0$, $y_1=1$, $x=y=3$ and $A=A_{16}$,
\\
$(iii)$: $x_1=y_1=0$, $x_2=y_2=0$, $x=y=3$ and $A=A_{17}$,
\\
$(iv)$: $x_1=y_1=0$, $x_2=x-1$, $y_2=1$, which leads to $A_{8}$ with $m\geq 2$, $A_{18}$, or $A_{19}$.
\end{prp}
\begin{proof}
Write $M_{4,4}=(c_1-2)J+(c_2-2)L+y_1 K + R$, where
\[ 
L=\left[
\begin{array}{rrc}
1&-1&0 \\ -1&1&0 \\ 0&0&0
\end{array}
\right],\ 
K=\left[\begin{array}{rrc}
0&0&0 \\ 0&1&1 \\ 0&1&1
\end{array}
\right],\ \mbox{ and }
R=\left[\begin{array}{ccc}
x+2&0&x_1'+1 \\ 0&y-y_1+2&2 \\ x_1'+1&2&x_1+1
\end{array}
\right].
\]
Clearly $J$, $K$ and $L$ are psd and $R$ is pd, unless $x_1=0$,
or $x_1=x_1'=x-1=1$ and $y-y_1=2$.
In the latter case 
$R$ is psd with kernel spanned by $[1, 1, -2]^\to$, which is in the kernel of $M_{4,4}$ only if $y_1=0$.
Therefore $x=y=2$, and we obtain the matrices $A_4$ with $m,\ell\geq 2$.
\\
So we have that $x_1=0$. 
If $y_1\geq 1$, then we use that $R+y_1 K$ is not pd.
By straightforward computation we get 
$\det(R+y_1 K)=(y+2)(x+2)(y_1+1)-(x+2)(y_1+2)^2-(y+2)$.
Using $y-y_1\geq 2$ this gives 
$\det(R+y_1 K)\geq (x+1)y_1-4$ and equality implies $y-y_1=2$.
Therefore $y_1=0$, or 
$y_1=1$, $y=3$ and $x\leq 3$, or
$y_1=2$, $y=4$ and $x=1$.
If $(x,x_1,y,y_1)=(3,0,3,1)$ then $R+y_1 K$
is psd and the kernel is spanned by $[1,3,-5]^\to$
which is in the kernel of $J$ nor $K$.
Hence $M_{4,4}$ is nonsingular, 
unless $c_1=c_2=2$ in which case we obtain $A_{16}$. 
It is not difficult to rule out the remaining cases
by use of Lemma~\ref{deg1}$(ii)$, 
Proposition~\ref{restrictions}$(ii)$, 
and the last forbidden subgraph of Figure~\ref{forbidden-all}.
\\
What remains is the case $x_1=y_1=0$.
Using Gaussian elimination in $M_{2,2}$ we find that in this case 
$\det(M_{2,2})=(y_2-1)(c_2+x-1)+(x-x_2-1)(c_2+x_2)$.
Note that $y_1=0$ implies that $x_2\leq x-1$.
Therefore $M_{2,2}$ is singular only if 
$(a)$: $y_2=0$ and $c_2+x-1=(x-x_2-1)(c_2+x_2)$,
or $(b)$: $y_2=1$, $x=x_2+1$.
In case $(a)$ we get $c_2=2$, $x=3$, $x_2=y_2=0$.
Since no two rows of $A$ are equal we have that 
there is exactly one vertex $v\in X$ for which $x_1=y_1=0$, and
at most one vertex $w\in Y$ for which $x_2=y_2=0$.
Using this observation we find that case $(a)$ leads to $A_{17}$, and that otherwise $(b)$ holds for all vertices of $Y$. 
Moreover, $\det(M_{1,2})=(1-c_1)(c_2+x-1)(y_2'+1)$, implies $y_2'=-1$.
From Proposition~\ref{restrictions}$(iii)$ we deduce that $G^\sigma_X$ is equal to $K_1$, or $K_1+K_{x-1}$ (all signs are positive),
or consists of one isolated vertex and $(x-1)/2$ isolated signed edges.
\\
If $G^\sigma_X=K_1$ then we have an equitable partition of $A$ with quotient matrix
\[
Q=\left[
\begin{array}{cccc}
c_1-1 & c_2   & 1 & 0 \\ 
c_1   & 1-c_2 & 0 & y \\
c_1   & 0     & 0 & 0 \\
0     & c_2   & 0 & -1
\end{array}
\right].
\]
By Gaussian elimination in $Q+I$ and $Q-I$ we find that 
$Q$ has two eigenvalues equal to $\pm 1$ only if
$c_1=2$ and $y=4$. 
This gives $A_{8}$ with $m\geq 2$.
\\
If $G^\sigma_X=K_1+K_{x-1}$ then 
we have an equitable partition of $A$ with quotient matrix
\[
Q=\left[
\begin{array}{ccccc}
c_1-1 & c_2   & 1 & x-1 & 0 \\ 
c_1   & 1-c_2 & 0 & 0 & y \\
c_1   & 0     & 0 & 0 & 0 \\
c_1   & 0     & 0 & x-2 & y \\
0     & c_2   & 0 & x-1 & -1
\end{array}
\right].
\]
By Gaussian elimination in $Q+I$ and $Q-I$ we find that 
$Q$ has at least three eigenvalues unequal to $\pm 1$. 
\\
If $G^\sigma_X$ has an negative isolated edge $e$ then $C_2=2$ by Theorem~\ref{restrictions}$(iv,v)$.
If $v$ is a vertex of $e$, then $c_2=2$ $x_1=-x_1'=1$, $y_1=y$, $x_2=x-1$, $y_2=-y_2'=1$ and
$\det(M_{2,2})=(x+1)(c_1 y-y-c_1)-4(c_1-1)=0$.
Using that $x>1$ and odd, and $y\geq 2$ and even,
we find just two solutions: $(c_1,x,y)=(3,5,2)$ and 
$(4,7,2)$ and obtain $A_{18}$.
\\
Assume $G^\sigma_X$ has a positive isolated edge $e$.
Take $v\in e$.
Then $x_1=x_1'=1$, $y_1=y$, $x_2=x-1$, $y_2=-y_2'=1$ and $\det(M_{2,2})=(c_2+x-1)(c_1 y-y-c_1)=0$,
which implies $c_1=y=2$.
Since $c_1\neq 3$ or $4$, $G^\sigma_X$ has no
negative isolated edge. 
This means that $A$ has an equitable partition
with quotient matrix
\[
Q=\left[
\begin{array}{ccccc}
1 & c_2   & 1 & x-1 & 0 \\ 
2 & 1-c_2 & 0 & 0   & 2 \\
2 & 0     & 0 & 0   & 0 \\
2 & 0     & 0 & 1   & 2 \\
0 & c_2   & 0 & x-1 & -1
\end{array}
\right].
\]
Using Gaussian elimination in $Q+I$ and $Q-I$ we find that $Q$ has three eigenvalues equal to $\pm 1$ if and only if $(c_2,x)=(3,7)$ or $(4,5)$.
Thus we obtain $A_{19}$. 
\end{proof}

Suppose the matrix $A$ does not belong to one of the cases of the above proposition. 
Then each vertex from $X$ is adjacent to $y$ or $y-1$ vertices of $Y$.  
And, by interchanging $C_1$ with $C_2$ and $X$ with $Y$ it also follows that each vertex from $Y$ is adjacent to $x$ or $x-1$ vertices of $X$.
So it follows that the adjacency matrix of $G^\sigma _{X\cup Y}$ takes the form:
\[
A_{X\cup Y}=
\left[ 
\begin{array}{cc}
A_X & B \\
B^\to & A_Y
\end{array}
\right] 
\mbox{ with }
B=
\left[ 
\begin{array}{cc}
J & J \\
J-I_m & J
\end{array}
\right] 
\mbox{ and }
0\leq m\leq\min\{x,y\}.
\] 
\begin{prp}
With $B$ as above, one of the following holds:
\\
$(i)$ $m=x=y$, $B=J-I$ which leads to $A_{11}$,
\\
$(ii)$ $m=0$, $B=J$ which leads to $A_2$, $A_{12}$, $A_{13}$, $A_{14}$, or $A_{15}$.
\end{prp}
\begin{proof}
Let $X_m$ and $Y_m$ be the subsets of $X$ and $Y$ corresponding to the submatrix $J-I_m$ of $B$.
From Proposition~\ref{restrictions}$(ii)$ it follows that $X_m$ and $Y_m$ consists of isolated vertices in $\Gamma^\sigma_X$ and $\Gamma^\sigma_Y$, respectively. 
So if $m=x=y$ then $\Gamma^\sigma_X$ and $\Gamma^\sigma_Y$
have no edges and $A$ has the block structure of $A_{12}$.
The values of $m$, $\ell$ and $k$ follow from the spectrum of the quotient matrix.
\\
If $0<m<x$ then the subgraph $\Gamma^\sigma_{X\setminus X_m}$ of 
$\Gamma^\sigma_X$ induced by $X\setminus X_m$ cannot have an isolated edge or vertex (because of Lemma~\ref{deg1}$(ii)$), 
and therefore $\Gamma^\sigma_{X\setminus X_m} = K_{x-m}$
(because of Proposition~\ref{restrictions}$(iv)$).
Similarly, if $0<m<y$ then $\Gamma^\sigma_{Y\setminus Y_m} = K^-_{y-m}$.
Thus, if $0<m<\min\{x,y\}$ then $A$ has an equitable partition 
with quotient matrix
\[
Q=\left[
\begin{array}{cccccc}
c_1-1 & c_2   & x-m   & m   & 0   & 0   \\ 
c_1   & 1-c_2 & 0     & 0   & m   & y-m \\
c_1   & 0     & x-m-1 & 0   & m   & y-m \\
c_1   & 0     & 0     & 0   & m-1 & y-m \\
0     & c_2   & x-m   & m-1 & 0   & 0   \\
0     & c_2   & x-m   & m   & 0   & m-y+1 
\end{array}
\right].
\]
However $Q$ does not have four eigenvalues equal to $\pm 1$.
A quick way to verify this is by Gaussian elimination in $Q+I$ with row~2 and column~6 deleted and in $Q-I$ with row~1 and column~3 deleted.
Then we find that $Q+I$ and $Q-I$ both have rank at least~5.
A similar argument works if $m=x<y$, or $m=y<x$. 
Then we have an equitable partition with a $5\times 5$ quotient matrix with at most two eigenvalues equal to $\pm 1$. 
Thus we can conclude that $m=0$ and $B=J$. 
Now $\Gamma^\sigma_X$ may have an isolated vertex, but 
Lemma~\ref{deg1} still implies that $\Gamma^\sigma_X$ cannot have 
two isolated vertices, or an isolated vertex and an isolated edge.
Therefore by use of Proposition~\ref{restrictions}$(iii)$ 
there are just a few possibilities for $\Gamma^\sigma_X$, being:
(a): $\Gamma^\sigma_X=K_x$ ($x\geq 3$),
(b): $\Gamma^\sigma_X=K_{x-1}+K_1$ ($x\geq 4$),
(c): $\Gamma^\sigma_X$ consists of $x/2$ isolated signed edges.
For $\Gamma^\sigma_Y$ we obtain the same list with
$y$, $K^-_y$ and $K^-_{y-1}$ instead of $x$, $K_x$ and $K_{x-1}$. 
\\
If $\Gamma^\sigma_X=K_x$ and $\Gamma^\sigma_Y=K^-_y$, then $A$ has an equitable partition of $A_{11}$, and the quotient matrix $Q$ has two eigenvalues equal to $\pm 1$ only in the given cases 
(taking Assumption \ref{asm} into account). 
If $\Gamma^\sigma_X=K_x$ and $\Gamma^\sigma_Y=K^-_{y-1}+K_1$
(or $\Gamma^\sigma_X=K_x+K_1$ and $\Gamma^\sigma_Y=K^-_{y}$),
then we have an equitable partition of $A$ with $5\times 5$ quotient matrix $Q$, and by Gaussian elemination it follows that $Q+I$ and $Q-I$
both have rank at least 4, and therefore $Q$ (and $A$) have at least three eigenvalues unequal to $\pm 1$.
If $\Gamma^\sigma_X=K_{x-1}+K_1$ and $\Gamma^\sigma_Y=K^-_{y-1}+K_1$,
then the equitable partition has a $6\times 6$ quotient matrix $Q$,
and both $Q+I$ and $Q-I$ both have rank at least 5. 
So also in this case we find no solution.  
For the remainder of the proof we may assume that $\Gamma^\sigma_X$ consists of $x/2$ disjoint signed edges. 
We claim that all these edges have the same sign. 
Indeed, suppose $\Gamma^\sigma_X$ has a negative and a positive edge, then we consider the matrix $N$ consisting of three rows
of $A$ corresponding to a vertex of the negative edge, a vertex of the positive edge and a vertex from $C_1$. 
Then 
\[
NN^\to-I=\left[ 
\begin{array}{ccc}
c_1+y & c_1+y & c_1\\
c_1+y & c_1+y & c_1-2\\
c_1   & c_1-2 & c_1+c_2+x-2\\
\end{array}
\right] 
\]
is nonsingular, which proves the claim.
\\
Suppose that all edges of $\Gamma^\sigma_X$ have a negative sign.
Then Proposition~\ref{restrictions}$(iv)$ implies that $c_2=2$.
If $\Gamma^\sigma_Y=K^-_y$ then $A$ has an equitable partition
with three parts and a quotient matrix with no eigenvalue equal to $\pm 1$.
If $\Gamma^\sigma_Y=K^-_{y-1}+K_1$ then $A$ has an equitable partition with four parts and quotient matrix
\[
Q=\left[
\begin{array}{cccc}
c_1-1 & x+2 & 0   & 0 \\ 
c_1   & -1  & y-1 & 1 \\
0     & x+2 & 2-y & 0 \\
0     & x+2 & 0   & 0 
\end{array}
\right].
\]
We easily obtain that $Q+I$ is singular only if $y=5$,
and that $Q-I$ is singular only if $(c_1,y)=(5,4)$ or $(6,2)$.
Thus we find $A_{13}$.
If $\Gamma^\sigma_Y$ consists of $y/2$ positive edges, then 
also $c_1=2$, and we obtain $A_2$.
If $\Gamma^\sigma_Y$ consists of $y/2$ negative edges,
then we find an equitable partition with quotient
\[
Q=\left[
\begin{array}{ccc}
c_1-1 & x+2 & 0    \\ 
c_1   & -1  & y \\
0     & x+2 & -1
\end{array}
\right].
\]
Now $Q+I$ is nonsingular and $Q-I$ is singular only if $(c_1,x,y)=(6,2,4)$, or $(5,4,4)$. 
Thus we find $A_{14}$.
\\
Next suppose that $\Gamma^\sigma_X$ consists of positive edges.
If $\Gamma^\sigma_Y=K^-_y$, then $A$ has an equitable partition for which the quotient matrix has at least three eigenvalues unequal to $\pm 1$.
If $\Gamma^\sigma_Y=K^-_{y-1}+K_1$, then $A$ has an equitable partition with five parts and quotient matrix $Q$ for which 
it is easily seen that both $Q+I$ and $Q-I$ have rank at least $4$, so we find no solution.
If $\Gamma^\sigma_Y$ has $y/2$ disjoint positive edges, then we obtain $A_{14}$ after by switching taking the negative.
If $\Gamma^\sigma_Y$ has $y/2$ disjoint negative edges,
then $A$ has an equitable partitions with quotient matrix 
\[
Q= \left[
\begin{array}{cccc}
c_1-1 & c_2   & x & 0 \\ 
c_1   & 1-c_2 & 0 & y \\
c_1   & 0     & 1 & y \\
0     & c_2   & x & -1
\end{array}
\right].
\]
Using that $x$ and $y$ are both even we find that
both $Q-I$ and $Q+I$ have rank~3 only if 
$(c_1,c_2,x,y)=(3,3,6,6)$, $(4,3,6,4)$, or $(4,4,4,4)$.
Thus we find $A_{15}$.
\end{proof}
\begin{ccl}
If $G^\sigma$ has adjacency matrix $A$ with Format $(iii)$ with $Z=\emptyset$ and $X,Y \neq\emptyset$ then 
$A$ equals 
$A_2$, $A_4$, $A_{8}$, $A_{12}$, $A_{13}$, $A_{14}$, $A_{15}$, $A_{16}$, $A_{17}$, $A_{18}$, or $A_{19}$. 
\end{ccl}

\section{Format(ii) with $X$, $Y$ nonempty and $Z$ empty}
In this section $G^\sigma$ has adjacency matrix $A$ with format~$(ii)$, $X, Y\neq\emptyset$ and $Z=\emptyset$.
We use the notation of the previous section.
Again the matrix $N$ consists of four rows of $A$ corresponding to a vertex from $C_1$, 
a vertex from $C_2$, a vertex $v$ from $X$ and a vertex from $w$ from $Y$. 
Then  
\[
M=NN^\to-I=
\left[
\begin{array}{cccc}
c_1+c_2+x-2  & c_1+c_2-2 & c_1+x_1'-1 & c_2+x_2 \\[3pt]
c_1+c_2-2 & c_1+c_2+y-2 & c_1+y_1 & c_2-1+y'_2 \\[3pt]
c_1+x_1'-1 & c_1+y_1 & c_1+x_1+y_1-1 & x_{1,2}+y_{1,2} \\[3pt]
c_2+x_2 & c_2-1+y'_2 & x_{1,2}+y_{1,2} & c_2+x_2+y_2-1
\end{array}
\right].
\]
First we consider the case that $c_1,c_2\geq 2$.
Then we can imitate the steps of the previous section.  

\subsection{$c_1\geq 2$ and $c_2\geq 2$}

\begin{prp}\label{12-1}  
Suppose $y-y_1\geq 2$ for some vertex in $X$, then $A=A_9$.
\end{prp}
\begin{proof}
From Proposition~\ref{restrictions}($iii,iv$) it follows that $x_1=-x_1'$, or $x_1=x_1'=1$.
We consider $M_{4,4}$ and write $M_{4,4}=(c_1-2)J+(c_2-2)L+y_1 K + R$, where
\[ 
L=\left[
\begin{array}{rrc}
1&1&0 \\ 1&1&0 \\ 0&0&0
\end{array}
\right],\ 
K=\left[\begin{array}{rrc}
0&0&0 \\ 0&1&1 \\ 0&1&1
\end{array}
\right],\ \mbox{ and }
R=\left[\begin{array}{ccc}
x+2 & 2 & x_1'+1 \\ 2 & y-y_1+2 & 2 \\ x_1'+1 & 2 & x_1+1
\end{array}
\right].
\]
Then $J$, $K$ and $L$ are psd.
Moreover, if $y-y_1\geq 3$, or $y-y_1=2$ and $x_1\neq 0$ then $R$ is pd. 
Using $y-y_1\geq 2$ we obtain $x_1= x_1'=0$ and $y-y_1 = 2$.
In this case $R$ is psd with kernel spanned by $[0,1,-2]^\to$.
So only if $c_1=c_2=2$ and $y_1=0$ the kernel of $M_{4,4}$ is nontrivial. 
This implies that $M_{4,4}$ is singular only if $(c_1,c_2,y,x_1,y_1)=(2,2,2,0,0)$.
If $X$ contains another vertex with $x_1=y_1=0$ then $A$ has two equal rows,
so if $x\geq 2$ then there must be 
another solution for $\det(M_{4,4})=0$
with $c_1=c_2=y=2$.
Such a solution exists only if $x_1=x_1'=1$ which leads to $A_9$.
\end{proof}
As in the previous section it follows from Proposition~\ref{12-1} that if $A\neq A_9$ then the adjacency matrix of $G^\sigma _{X\cup Y}$ takes the form:
\[
A_{X\cup Y}=
\left[ 
\begin{array}{cc}
A_X & B \\
B^\to & A_Y
\end{array}
\right] 
\mbox{ with }
B=
\left[ 
\begin{array}{cc}
J & J \\
J-I_m & J
\end{array}
\right] 
\mbox{ and }
0\leq m\leq\min\{x,y\},
\] 
and with the same arguments, but (slightly) different quotient matrices we obtain the following result:
\begin{prp}
The matrix $B$ above has $m=0$
and $A$ equals $A_8$ with $m=1$, $A_{12}$
(only the case $k=\ell$ because of Assumption~\ref{asm}), or $A$ is an unsigned matrix given in case~$(iii)$ of Theorem~1 from~\cite{CHVW}.
\end{prp}

\subsection{$c_1=1$ and $c_2\geq 2$}

Since $c_1=1$, Assumption~\ref{asm} yields that $Y$ is a coclique, so $y_2=y_2'=y_{1,2}=0$, and $M$ becomes:
\[
M=
\left[
\begin{array}{cccc}
c_2+x-1  & c_2-1 & x_1' & c_2+x_2  \\[3pt]
c_2-1 & c_2+y-1 & y_1+1 & c_2-1  \\[3pt]
x_1' & y_1+1 & x_1+y_1+z_1 & x_{1,2} \\[3pt]
c_2+x_2 & c_2-1 & x_{1,2} & c_2+x_2-1
\end{array}
\right].
\]
Suppose vertex $v$ of $G^\sigma_X$ is isolated.
Then 
$x_1=x_1'=x_{1,2}=0$ and 
$\det(M_{3,3})=-((c_2+x_2)y+(x_2+1)(c_2-1))\neq 0$. 
Hence $G^\sigma_X$ has no isolated vertex and from Proposition~\ref{restrictions}$(i,iii,iv)$ it follows that there are just three possibilities: 
$(i)$ $G^\sigma_X$ is the complete graph with all edges negative,
$(ii)$ $G^\sigma_X$ is the disjoint union of two or more negative edges, and 
$(iii)$ $c_2=2$ and $G^\sigma_X$ is the disjoint union of two or more signed edges (not all negative).
Assume we have case $(i)$, then $x_1 = -x_1' = x-1 > 0$ and Proposition~\ref{restrictions}$(ii)$ 
implies that every vertex in $Y$ is adjacent to all or no vertices of $X$.
If $w\in Y$ is adjacent to all of $X$, then $x_2=x$ and 
find $\det(M_{3,4})=(x-1)(c_2+y-1)+(c_2-1)(y_1+1)>0$.
If every vertex of $Y$ is nonadjacent to every vertex of $X$, then $y_1=x_2=0$ and we find $\det(M_{3,4})=(x-1)yc_2>0$.
So we can conclude that case $(i)$ does not occur.
For the remaining two cases $G^\sigma_X$ is a disjoint union of edges. 
Consider an edge $e$ of $G^\sigma_X$.
By Proposition~\ref{restrictions}$(ii)$ every vertex in $G^\sigma_Y$ is adjacent to none or both vertices of $e$.
Therefore $x_2$ and $x-x_2$ are even.
Next consider case $(ii)$. 
Suppose $\{v,w\}$ is an edge between $X$ and $Y$, then 
$x_1'= \x_{1,2}=-x_1=-1$ and we find 
$\det(M_{3,4})= (x-x_2-1)(-(c_2+y-1)-(c_2-1)(y_1+1))\neq 0$.
Therefore there are no edges between $X$ and $Y$. 
So $x_2=y_1=0$, $x_1=-x_1'=1$ and 
$\det(M_{3,1})=-y(c_2-1)-(c_2-1)$, which implies $y=1$.
Now $\det(M_{1,4}) = 0$ gives $x=4$ and $c_2=2$
and thus we find that $A$ equals $A_8$ with $m=1$.
In case $(iii)$ we have $c_2=2$ and consider $M_{3,3}$.
It is straigtforward that $\det(M_{3,3})\neq 0$, except when $x=4$, $y=1$ and $x_2=0$ which again leads to $A_8$ with $m=1$.
Thus we conclude that the last case gives no new examples.

{\begin{ccl}
If $A$ has Format $(ii)$ with $Z=\emptyset$ and $X,Y \neq\emptyset$ then 
$A$ is equal to $A_8$ with $m=1$, $A_{9}$, $A_{12}$ with $k=\ell$, or $A$ is an unsigned matrix given in Theorem~1$(iii)$  of~\cite{CHVW}.
\end{ccl}

\section{Recapitulation} 
By combining the conclusions of sections 6, and 8 to 12 we have completed the proof of Theorem~\ref{main}. 
Note that the proof doesn't exclude the unsigned examples. 
Thus we rediscovered the unsigned characterization from \cite{CHVW}.
As in the unsigned case, with the present characterization one can examine which signed graphs in $\SG$ are determined, up to switching, by the spectrum. 
This however, is more involved as in the unsigned case and will be the subject of further research.
\\[3pt]

\noindent
{\bf\Large Acknowledgement}
\\[-3pt]

\noindent
Part of the research was done while the first author visited the Haci Bekta\c{s} Veli University of Nev\c{s}ehir on a grant from T\"ubitac.

\end{document}